\documentclass[12pt]{article} 

\usepackage{amsmath}
\usepackage{amsthm}
\usepackage{amsfonts}
\usepackage{mathrsfs}
\usepackage{stmaryrd}
\usepackage{setspace}
\usepackage{fullpage}
\usepackage{amssymb}
\usepackage{breqn}
\usepackage{enumitem}
\usepackage{bbold} 
\usepackage{comment}
\usepackage{hyperref}
\usepackage{pgf,tikz}
\usepackage{graphicx}
\usepackage{colonequals}
\usepackage{microtype}
\usetikzlibrary{decorations.pathreplacing,arrows}
\tikzstyle{vertex}=[circle,draw=black,fill=black,inner sep=0,minimum size=3pt,text=white,font=\footnotesize]

\newtheorem*{thm*}{Theorem}
\newtheorem{thm}{Theorem}%[chapter]

\newtheorem{lemma}[thm]{Lemma}

\newtheorem{proposition}[thm]{Proposition}

\newtheorem{definition}[thm]{Definition}

\newtheorem{corollary}[thm]{Corollary}
\newtheorem*{corollary*}{Corollary}

\newtheorem{clm}[thm]{Claim}
\newtheorem*{proposition*}{Proposition}
\newtheorem{observation}[thm]{Observation}

\newcommand\ex{\ensuremath{\mathrm{ex}}}

\newcommand\cC{{\mathcal C}}

\newcommand\cF{{\mathcal F}}

\newcommand\cH{{\mathcal H}}

\newcommand\cK{{\mathcal K}}

\newcommand\cR{{\mathcal R}}

\newcommand{\ignore}[1]{}

\if11     % 11 = on,  10 = off
\usepackage[mathlines]{lineno}
\newcommand*\patchAmsMathEnvironmentForLineno[1]{%
  \expandafter\let\csname old#1\expandafter\endcsname\csname #1\endcsname
  \expandafter\let\csname oldend#1\expandafter\endcsname\csname end#1\endcsname
  \renewenvironment{#1}%
     {\linenomath\csname old#1\endcsname}%
     {\csname oldend#1\endcsname\endlinenomath}}% 
\newcommand*\patchBothAmsMathEnvironmentsForLineno[1]{%
  \patchAmsMathEnvironmentForLineno{#1}%
  \patchAmsMathEnvironmentForLineno{#1*}}%
\AtBeginDocument{%
\patchBothAmsMathEnvironmentsForLineno{equation}%
\patchBothAmsMathEnvironmentsForLineno{align}%
\patchBothAmsMathEnvironmentsForLineno{flalign}%
\patchBothAmsMathEnvironmentsForLineno{alignat}%
\patchBothAmsMathEnvironmentsForLineno{gather}%
\patchBothAmsMathEnvironmentsForLineno{multline}%
}
\fi

\def\inst#1{$^{#1}$}

\begin{document}

\title{Hypergraph based Berge hypergraphs}

\author{Martin Balko\inst{1}\thanks{The first author was supported by the grant no.~19-04113Y of the Czech Science Foundation (GA\v{C}R) and by the Center for Foundations of Modern Computer Science (Charles University project UNCE/SCI/004).} 
\and
D\'aniel Gerbner\inst{2}\thanks{
The second author was supported in part by the J\'anos Bolyai Research Fellowship of the Hungarian Academy of
Sciences and the National Research, Development and Innovation Office - NKFIH under the grants K
116769, KH 130371 and SNN 12936.}
\and 
Dong Yeap Kang\inst{3}\thanks{The third author was supported by the National Research Foundation of Korea (NRF) grant funded by the
Korea government (MSIT) (No. NRF-2017R1A2B4005020) and also by TJ Park Science Fellowship of POSCO
TJ Park Foundation. }
\and
Younjin Kim\inst{4}\thanks{The fourth author is a corresponding author and was supported by Basic Science Research Program through the National Research Foundation of Korea (NRF) funded by the Ministry of Education (2017R1A6A3A04005963).}
\and
Cory Palmer\inst{5}
}

\maketitle

\begin{center}
{\footnotesize
\inst{1} 
Department of Applied Mathematics, \\
Faculty of Mathematics and Physics, Charles University, Czech Republic \\
\texttt{balko@kam.mff.cuni.cz}
\\\ \\
\inst{2} 
Alfr\'{e}d R\'{e}nyi Institute of Mathematics, Hungarian Academy of Sciences, Budapest, Hungary \\
\texttt{gerbner@renyi.hu}
\\\ \\
\inst{3} 
Department of Mathematical Sciences, KAIST, South Korea \\
\texttt{dyk90@kaist.ac.kr}
\\\ \\
\inst{4} 
Institute of Mathematical Sciences, Ewha Womans University, Seoul, South Korea \\
\texttt{younjinkim@ewha.ac.kr}
\\\ \\
\inst{5} 
University of Montana, Missoula, Montana 59812, USA \\
\texttt{cory.palmer@umontana.edu}
}
\end{center}

\begin{abstract}
Fix a hypergraph $\mathcal{F}$. A hypergraph $\mathcal{H}$ is called a {\it Berge copy of $\mathcal{F}$} or {\it Berge-$\mathcal{F}$} if we can choose a subset of each hyperedge of $\mathcal{H}$ to obtain a copy of $\mathcal{F}$.  A hypergraph $\mathcal{H}$ is {\it Berge-$\mathcal{F}$-free} if it does not contain a subhypergraph which is Berge copy of $\mathcal{F}$. This is a generalization of the usual, graph based Berge hypergraphs, where $\cF$ is a graph. 

In this paper, we study extremal properties of hypergraph based Berge hypergraphs and generalize several results from the graph based setting. In particular, we show that for any $r$-uniform hypregraph $\mathcal{F}$, the sum of the sizes of the hyperedges of a (not necessarily uniform) Berge-$\mathcal{F}$-free hypergraph $\mathcal{H}$ on $n$ vertices is $o(n^r)$ when all the hyperedges of $\mathcal{H}$ are large enough. We also give a connection between hypergraph based Berge hypergraphs and generalized hypergraph Tur\'an problems.
\end{abstract}

\section{Introduction}

Berge \cite{berge} defined hypergraph cycles in the following way. A cycle of length $k$ corresponds to $k$ vertices $v_1,\dots, v_k$ and $k$ distinct hyperedges $h_1,\dots, h_k$ such that $v_i$ and $v_{i+1}$ are both contained in $h_i$ for $i<k$, while $h_k$ contains both $v_k$ and $v_1$. Observe that for 2-uniform hypergraphs these are the usual graph cycles, but in case of larger hyperedges there are several non-isomorphic cycles of length $k$. Note that there are several other definitions of hypergraph cycles; these ones are usually referred to as \textit{Berge cycles}.

Berge did not study extremal properties of these cycles. Lazebnik and Verstra{\"e}te \cite{lazver} were the first to consider such problems. They gave bounds on the number of hyperedges in an $r$-uniform hypergraph without a Berge cycle of length at most 4. Many similar results followed  for Berge cycles (see, for example, \cite{gyori,gyl}) and Berge paths \cite{gykl}, where a \emph{Berge path} is obtained from a Berge cycle by removing a hyperedge (similarly for the graph case).

Gerbner and Palmer \cite{gp1} observed that the way we obtain a Berge cycle from a graph cycle $C_k$ or a Berge path from a path $P_k$ can be generalized to any graph $G$. 

\begin{definition} A hypergraph $\mathcal{H}$ is a \emph{Berge copy of a graph $G$} (in short: \emph{Berge-$G$}) if $V(G)\subseteq V(\mathcal{H})$ and there is a bijection $f:E(G)\rightarrow E(\mathcal{H})$ such that for any $e\in E(G)$ we have $e\subseteq f(e)$.
\end{definition}

In other words, $\mathcal{H}$ is Berge-$G$ if we can choose a pair of vertices (i.e., a graph edge) in each hyperedge of $\mathcal{H}$ to obtain a copy of $G$. The maximum number (or weight) of hyperedges in Berge-$G$-free hypergraphs has been studied by a number of authers (see, for example, \cite{sgmp,gmp,gmv,pttw,t}, and Subsection 5.2.2. of \cite{gp} for a short survey). Others have investigated saturation problems \cite{linsat,egkms}, Ramsey problems \cite{agy,gmov,STZZ} or spectral radius \cite{kllw} for Berge hypergraphs.

It has been observed in several different settings \cite{anssal,linsat,gp,STZZ} that one can analogously define Berge copies of a hypergraph. More precisely we have the following definition.

\begin{definition} Fix a hypergraph $\mathcal{F}$. A hypergraph $\mathcal{H}$ is a \emph{Berge copy of $\mathcal{F}$} (in short: \emph{Berge-$\mathcal{F}$}) if $V(\mathcal{F})\subseteq V(\mathcal{H})$ and there is a bijection $f:E(\mathcal{F})\rightarrow E(\mathcal{H})$ such that for any $h\in E(\mathcal{F})$ we have $h\subseteq f(h)$. In the case when $\mathcal{H}$ is $k$-uniform, we say $\mathcal{H}$ is \emph{Berge$_k$-$\mathcal{F}$}.

Fix a family of hypergraphs $\cC$. A hypergraph $\mathcal{H}$ is a \emph{Berge copy of $\cC$} (in short: \emph{Berge-$\cC$}) if it is a Berge copy of some member $\mathcal{F}\in \cC$. In the case when $\mathcal{H}$ is $k$-uniform, we say $\mathcal{H}$ is \emph{Berge$_k$-$\cC$}.
\end{definition}

We call such hypergraphs \textit{hypergraph based Berge hypergraphs}. So far there has been no systematic study of their properties.
In this paper we focus on the maximum number (or sum of weights) of edges in hypergraphs avoiding Berge copies of a given hypergraph $\mathcal{F}$. Note that the Berge-$\mathcal{F}$ and Berge$_k$-$\mathcal{F}$ are defined even if $\mathcal{F}$ is not $r$-uniform for some $r \leq k$. However, here we will only deal with the case when $\mathcal{F}$ is $r$-uniform.

As we often consider several hypergraphs simultaneously, for the sake of brevity we will use the term \emph{$r$-graph} in place of $r$-uniform hypergraph. The term \emph{$r$-edge} refers to a hyperedge of an $r$-graph.
Given a family $\cC$ of $r$-graphs, we denote by $\ex_r(n,\cC)$ the maximum possible number of hyperedges in an $n$-vertex $r$-graph that does not contain any member of $\cC$ as a sub-hypergraph. Thus, when $\cC$ is the family Berge-$\mathcal{F}$ for some hypergraph $\mathcal{F}$, we use the notation $\ex_r(n,\textrm{Berge-}\mathcal{F})$. 

\subsection{Graph based Berge hypergraphs}

In this subsection we state several known results for the ordinary, graph based Berge hypergraphs, that we will generalize. First we deal with Berge-$F$-free hypergraphs that are not necessarily uniform, where $F$ is a graph. Observe that replacing a hyperedge with a larger hyperedge cannot destroy a copy of Berge-$F$, thus in order to maximize hyperedges in a Berge-$F$-free hypergraph it is preferable to use small hyperedges. Instead, we will assign a weight to each hyperedge that depends on the size of the hyperedge. Such problems were studied by Gy\H ori \cite{gyori} for triangles, by Gy\H ori and Lemons \cite{gyl} for cycles, and by Gerbner and Palmer \cite{gp1}  for arbitrary graphs. In these papers the weight of a hyperedge $h$ is $|h|$ or $|h|-c$ for some constant $c$.
Recently English, Gerbner, Methuku and Palmer \cite{sgmp} considered more general weight functions.

Before stating their results, we state a result of Gr\'osz, Methuku, and Tompkins~\cite{GMT} that deals with $k$-uniform Berge-$F$-free hypergraphs (where $F$ is a graph), for $k$ large enough.
Let us denote by $R^{(r)}(\cF,\cK)$ the two-color $r$-uniform Ramsey number of $r$-graphs $\cF$ and $\cK$. When $r=2$, put $R(F,K) = R^{(2)}(F,K)$. 
%If every hyperedge in $\mathcal{H}$ has size at least $ |V(F)|$, then we have $\sum_{h \in E(\mathcal{H})} |h| \ = o(n^2)$. For uniform Berge-F-free hypergraphs, Gr\'osz, Methuku, and Tompkins~\cite{GMT} proved that for any graph $F$ and $r$-uniform Berge $F$-free hypergraph $\mathcal{H}$ we have $|E(\mathcal{H})| \ = o(n^2)$ when $r$ is large enough as follows.

\begin{thm}[Gr\'osz, Methuku, Tompkins~\cite{GMT}]\label{grmeto}
Let $F$ be a graph and let $\mathcal{H}$ be a $k$-uniform Berge-$F$-free hypergraph. If $k \geq R(F,F)$, then 
\[
|E(\mathcal{H})| = o(n^2).
\]
\end{thm}

%Note that we have non-uniform hypergraphs, after adding up the sizes of the hyperedges. 
Using ideas from Gr\'osz, Methuku, and Tompkins~\cite{GMT},  English, Gerbner, Methuku, and Palmer~\cite{sgmp} showed that for any fixed graph $F$ the sum of the sizes of the hyperedges of a Berge-$F$-free hypergraph on $n$ vertices is $o(n^2)$ when all the hyperedges are large enough. 
%They~\cite{sgmp} also obtained the following theorem.
This follows from the two results below.

\begin{thm}[English, Gerbner, Methuku, Palmer~\cite{sgmp}] \label{EGMP}
Let $F$ be a fixed graph and let $\mathcal{H}$ be a Berge-$F$-free hypergraph on $n$ vertices.
If every hyperedge  of $\mathcal{H}$ has size at least $|V(F)|$, then
\[
\sum_{h\in E(\mathcal{H})} |h|^2 = O(n^2).
\]
Furthermore, if every hyperedge of $\mathcal{H}$ has size at least $R(F,F)$ and at most $o(n)$, then 
\[\sum_{h\in E(\mathcal{H})} |h|^2 = o(n^2).
\]
\end{thm}

\begin{corollary}
[English, Gerbner, Methuku, Palmer~\cite{sgmp}] \label{EGMP2}
Let $F$ be a fixed graph and let $\mathcal{H}$ be a Berge-$F$-free hypergraph on $n$ vertices.
Let  $w:\mathbb{Z}_+\rightarrow \mathbb{Z}_+$ be a weight function such that $w(m)=o(m^2)$.
If every hyperedge  of $\mathcal{H}$ has size at least $R(F,F)$, then 
\[
\sum_{h\in E(\mathcal{H})} w(|h|) = o(n^2).
\]
\end{corollary}

For uniform Berge-$F$-free hypergraphs, 
Gerbner and Palmer \cite{gp2} established a connection to generalized Tur\'an problems.
Given two graphs $H$ and $F$, let $\ex(n,H,F)$ denote the maximum number of copies of $H$ in an $F$-free graph on $n$ vertices. More formally, if $N(H,G)$ denotes the number of subgraphs of $G$ that are isomorphic to $H$, then 
\[
\ex(n,H,F)=\max\{N(H,G):\, G\, \text{is an $F$-free graph on $n$ vertices}\}.
\]

When $H=K_2$, this is the ordinary Tur\'an function $\ex(n,F)$. The first result concerning other graphs $H$ was the exact determination of $\ex(n,K_k,K_r)$ for each $k$ and $r$ by Zykov \cite{zykov}. Several more sporadic results followed it, and the systematic study of $\ex(n,H,F)$ was initiated by Alon and Shikhelman \cite{as}. The area is sometimes referred to as \textit{generalized Tur\'an problems}.

A connection between (graph based) Berge hypergraphs and generalized Tur\'an problems was obtained by Gerbner and Palmer \cite{gp2}, who proved that for any graph $F$, any $r$, and any $n$, we have 
\[
\ex(n,K_r,F)\le \ex_r(n,\textup{Berge-}F)\le \ex(n,K_r,F)+\ex(n,F).
\]
This was later generalized by Gerbner, Methuku, and Palmer \cite{gmp}. To state their result, we need to introduce some definitions.
We say that a graph $G$ is \emph{red-blue} if each of its edges is colored with one of the colors red and blue. Then for a red-blue graph $G$, we denote by $G_{red}$ the subgraph spanned by the red edges and $G_{blue}$ the subgraph spanned by the blue edges. Put $g_r(G)=|E(G_{red})|+N(K_r,G_{blue})$.

\begin{lemma}[Gerbner, Methuku, Palmer \cite{gmp}]\label{celeb2} For any graph $F$ and integers $r,n$ there is a red-blue $F$-free graph $G$ on $n$ vertices, such that $\ex_r(n,\textup{Berge-}F)\le g_r(G)$.
\end{lemma}

This lemma was used to prove several new bounds on $\ex_r(n,\textrm{Berge-}F)$ for various graphs $F$. Note that an essentially equivalent version was obtained by F\"uredi, Kostochka, and Luo~\cite{fkl}.

\section{Hypergraph based Berge hypergraphs}

The main results of this paper are hypergraph analogues of Theorem~\ref{EGMP}, Corollary~\ref{EGMP2}, and Lemma~\ref{celeb2}.

\begin{thm}\label{thm1}
Let $\cF$ be an $r$-graph with $r \geq 2$ and let $\cH$ be a %not necessarily uniform 
Berge-$\cF$-free hypergraph on $n$ vertices. If  every hyperedge of $\mathcal{H}$ has size at least $|V(\cF)|$, then 
\[
\sum_{h\in E(\mathcal{H})} |h|^r = O(n^r).
\]
Furthermore, if every hyperedge of $\mathcal{H}$ has size at least $R^{(r)}(\cF,\cF)$ and at most $o(n)$, then 
\[
\sum_{h\in E(\mathcal{H})} |h|^r = o(n^r).
\]
\end{thm}

\begin{corollary}\label{cor1}
Let $\cF$ be an $r$-graph with $r \geq 2$ and let $\mathcal{H}$ be a Berge-$\cF$-free hypergraph on $n$ vertices.
Let  $w:\mathbb{Z}_+\rightarrow \mathbb{Z}_+$ be a weight function such that $w(m)=o(m^r)$.
If every hyperedge  of $\mathcal{H}$ has size at least $R^{(r)}(\cF,\cF)$, then 
\[
\sum_{h\in E(\mathcal{H})} w(|h|) = o(n^r).
\]
\end{corollary}

Note that the above corollary is sharp in the sense that we cannot take a larger weight function as the hypergraph of a single hyperedge of size $\Omega(n)$ has weight $\Omega(n^r)$. We will prove Theorem~\ref{thm1} and Corollary~\ref{cor1} in Section~\ref{thm7pf}.
Another corollary is the following hypergraph analogue of Theorem~\ref{grmeto}.

\begin{corollary}\label{cor2} Let $\cF$ be an $r$-graph with $r \geq 2$ and let $\mathcal{H}$ be a $k$-uniform Berge-$\cF$-free hypergraph on $n$ vertices. If $k \geq R^{(r)}(\cF,\cF)$, then we have 
\[
|E(\mathcal{H})| = o(n^r).
\]
\end{corollary}

In the above corollaries the upper bound cannot be improved by much, even if the threshold $R^{(r)}(\cF,\cF)$ on the sizes of the hyperedges is increased (as long as the threshold does not depend on $n$). Indeed, let $\cF$ be an $r$-graph consisting of three hyperedges on $r+1$ vertices. Then a $k$-uniform Berge-$\cF$ has three hyperedges on at most $3(k-r)+k+1$ vertices. Alon and Shapira \cite{alsh} constructed for any $2\le r<k$ a $k$-uniform hypergraph with the property that any $3(k-r)+k+1$ vertices span less than $3$ edges, with $n^{r-o(1)}$ hyperedges. These hypergraphs are obviously Berge-$\cF$-free, giving the lower bound of $n^{r-o(1)}$ for this particular $\cF$ in both corollaries.

%\texttt{TODO: mention results from Sections 4 and 5}
\smallskip

Before stating our analogue of Lemma~\ref{celeb2} we need to introduce further definitions.
Given two $k$-graphs $\cH$ and $\cF$, let $\ex_k(n,\cH,\cF)$ denote the maximum number of copies of $\cH$ in $\cF$-free $k$-graphs on $n$ vertices. Let $N(\cH,\cH')$ denote the number of subhypergraphs of $\cH'$ that are isomorphic to $\cH$. Let $\cK_s^{(k)}$ be the complete $k$-graph on $s$ vertices (containing all possible $\binom{s}{k}$ hyperedges). 
%Note that it is a very different hypergraph from $K_r^{+k}$, which has only $\binom{r}{2}$ hyperedges.

As in the graph case we say that a hypergraph $\cH$ is \textit{red-blue} if each of its hyperedges is colored with one of the colors red and blue. Let $\cH_{red}$ and $\cH_{blue}$ be the subhypergraphs of $\cH$ spanned by the red hyperedges and the blue hyperedges of $\cH$, respectively. Put $g_r^k(\cH)=|E(\cH_{red})|+N(\cK_r^{(k)},\cH_{blue})$.

\begin{lemma}\label{celebhyp}  
For any $k$-graph $\cF$ and integers $r,n$ there is a red-blue $\cF$-free $k$-graph $\cH$ on $n$ vertices such that $\cH_{red}$ is $\cK_r^{(k)}$-free and $\ex_r(n,\textup{Berge-}\cF)\le g_r(\cH)$.
\end{lemma}

Note that the lemma implies 
\[
\ex_k(n,\cK_r^{(k)},\cF)\le\ex_r(n,\textup{Berge-}\cF)\le\ex_k(n,\cK_r^{(k)},\cF)+\ex_k(n,\cF).
\]

We prove Lemma~\ref{celebhyp} in Section~\ref{berge-section}. We also obtain some simple results in case $\cF$ itself is a graph based hypergraph (for example $\cF$ is a specific Berge copy of a graph $F_0$). Recall that the \emph{expansion} $F_0^{+k}$ of a graph $F_0$ is the Berge copy of $F_0$ which is constructed by adding $r-2$ new and distinct vertices to each edge of $F_0$. 
Using Lemma \ref{celebhyp} we can establish asymptotics for $\ex_k(n,\cK_r^{(k)},F_0^{+k})$ in case $\chi(F_0)>r>k$.

\section{Proof of Theorem~\ref{thm1}}\label{thm7pf}

\begin{proof}[Proof of Theorem~\ref{thm1}]
For $r \geq 2$, let us consider  the $r$-graph $\Gamma^{(r)}(S)$ with vertex set $S$ and whose edge set is the collection of all subsets of $S$ of size $r$. For a hypergraph $\mathcal{H}$, its \emph{$r$-shadow}, 
$\Gamma^{(r)}(\mathcal{H})$, is the $r$-graph with vertex set $V(\Gamma^{(r)}(\mathcal{H})) \colonequals V(\mathcal{H})$ and whose edge set is the set of all $r$-tuples contained in at least one hyperedge of $\mathcal{H}$.
 % Let $\mathcal{K}^{(r)}_s$ denote the complete $r$-uniform hypergraph on $s$ vertices.
 %  and let $\mathcal{B}_r(F)$ be the family of $r$-uniform Berge $F$ hypergraphs. 
 %Now we call an $r$-edge in $\Gamma^{(r)}(\mathcal{H})$ blue if it is contained in at most $|E(F)|-1$ hyperedges of $\mathcal{H}$. 
  We will use the following theorem of de Caen \cite{deca} about the Tur\'an number  of $\mathcal{K}^{(r)}_s$ for $s > r >2$.

  \begin{thm}[de Caen \cite{deca}]\label{dec}  For $s > r > 2$, 
  \[
  \ex_r(n,\mathcal{K}^{(r)}_s )\leq \binom{n}{r} \left( 1 - \frac{n-r+1}{ (n-s+1) \binom{s-1}{r-1}} \right) .
  \]
  \end{thm}

More precisely, we will use the following simple corollary: for $n \geq s$, there is an $\alpha=\alpha(r,s)<1$, such that $\ex_r(n,\mathcal{K}^{(r)}_s )\le \alpha\binom{n}{r}$.
This also holds for $r=2$ by Tur\'an's theorem.

%\begin{proof}[Proof of Theorem~\ref{thm1}]

Let $\cF$ be an $r$-graph and $\mathcal{H}$ be Berge-$\cF$-free hypergraph on $n$ vertices. We call an $r$-edge in $\Gamma^{(r)}(\mathcal{H})$ {\it blue} if it is contained in at most $|E(\cF)|-1$ hyperedges of $\mathcal{H}$.
 %Using the de Caen's result, we have the following bounds on the number of blue $r$-edges in a hyperedge of $\mathcal{H}$.
  \begin{clm}\label{claim1} There exists a constant $\alpha = \alpha(r,|V(\cF)|) <1$ such that for any hyperedge $h \in E(\mathcal{H})$, 
  the number of blue  $r$-edges in $\Gamma^{(r)}(h)$ is at least 
 $ (1- \alpha) { \binom{|h|}{r}} $.
 \end{clm}
 \begin{proof}[Proof of Claim~\ref{claim1}]
 %Let us consider a copy of $\cF$ in  $\Gamma^{(r)}(\mathcal{H})$. 
 We first show that every copy of~$\cF$ in $\Gamma^{(r)}(\mathcal{H})$ contains a blue $r$-edge. Suppose otherwise for the sake of a contradiction. Then, by the definition of a blue $r$-edge, every $r$-edge in a copy of $\cF$ is contained in at least $|E(\cF)|$ hyperedges of~$\mathcal{H}$. Therefore, we can construct greedily a copy of a Berge-$\cF$ in $\mathcal{H}$ by choosing, for each $r$-edge $f$ of $\cF$ a unique hyperedge of $\cH$ that contains $f$; a contradiction.  
 
 Fix a hyperedge $h \in E(\cH)$ and consider the subhypergraph of $\Gamma^{(r)}(h)$ formed by non-blue $r$-edges. % which are $\cF$-free.
 By the above argument, this subhypergraph of $\Gamma^{(r)}(h)$ is $\cF$-free.
 By Theorem~\ref{dec} there exists a constant $\alpha = \alpha(r,|V(\cF)|) < 1$ such that 
  the number of non-blue $r$-edges in $\Gamma^{(r)}(h)$ is at most $\alpha  \binom{|h|}{r}$, as $|h| \geq |V(\cF)|$.
 Thus, the number of blue $r$-edges in $\Gamma^{(r)}(h)$ is at least $ (1-\alpha)  \binom{|h|}{r}$.
\end{proof}

By applying Claim~\ref{claim1}, we obtain the following inequality.

\begin{align*}
\sum_{h \in E(\mathcal{H})} (1-\alpha)  \binom{|h|}{r} \leq  |\{ \text{blue } r\text{-edges in }  \Gamma^{(r)}(\mathcal{H})\}| \cdot (|E(F)|-1) = O(n^r).
\end{align*}

%\noindent  where $F$ is an $r$-uniform hypergraph.\\
This implies that 
\begin{equation}\label{part1}
 \sum_{h \in E(\mathcal{H})} |h|^r = O(n^r)   
\end{equation}
which proves the first part of Theorem~\ref{thm1}.

%Let  $\delta$  denote  the small number satisfying $n^{\delta} = o (n^{\frac{r-1}{r+1}})$.  Suppose that every $r$-edge of $\mathcal{H}$ has size at least $R^{(r)}(F,F)$. Since $|h| \leq |h|^r$ and $n^{\delta} = o (n^{\frac{r-1}{r+1}})$, we apply Lemma~\ref{lem2} to obtain the following equation.

%\begin{equation}\label{equa1}
%\sum_{h\in E(\mathcal{H}), \  |h| \leq n^{\delta}} |h| \  = o(n^r).
%\end{equation}

%\noindent Using Lemma~\ref{lem1}, we have

%\begin{align*}
%n^{\delta} \sum_{h\in E(\mathcal{H}), \  |h| >  n^{\delta}} |h| \ \  <  \sum_{h\in E(\mathcal{H}), \  |h| >  n^{\delta}} |h|^{r+\delta}  = o(n^{r+\delta}). \\
%\end{align*}

%\noindent Therefore we observe that 

%\begin{equation}\label{equa2}
%\sum_{h\in E(\mathcal{H}), \  |h| >  n^{\delta}} |h| \ = o(n^{r}).\\
%\end{equation}

%\noindent Now we are ready for our final argument by (\ref{equa1}) and (\ref{equa2}).

%\begin{align*}
%\sum_{h\in E(\mathcal{H})} |h| \  \ = 
%\sum_{h\in E(\mathcal{H}), \  |h| >  n^{\delta}} |h| \ \ + \sum_{h\in E(\mathcal{H}), \  |h| \leq n^{\delta}} |h| \ \ = o(n^r). 
%\end{align*}

It remains to prove the second part of Theorem~\ref{thm1}.
We may now assume that every hyperedge of $\cH$ has size at least $R^{(r)}(\cF,\cF)$ and at most $o(n)$.

\begin{clm}\label{claim2}
The number of copies of $\cF$ in $\Gamma^{(r)}(\mathcal{H})$ is $o(n^{|V(\cF)|})$.
\end{clm}

\begin{proof}[Proof of Claim~\ref{claim2}]
Since $\mathcal{H}$ is Berge-$\cF$-free, any copy of $\cF$ in $\Gamma^{(r)}(\mathcal{H})$ has at least two $r$-edges contained in a common hyperedge $h$ of $\mathcal{H}$. Thus at least $r+1$ vertices of $\cF$ are contained in $h$ and $|h| \geq r+1$. 
Now let us estimate the number of vertex sets that can span a copy of $cF$ as follows.
First pick a hyperedge $h$ of $\cH$, then pick $r+1$ vertices in $h$, and then pick $|V(\cF)|-r-1$ other vertices. 
Now, there are at most $|V(\cF)|!$ copies of $\cF$ with the same vertex set. Thus
\begin{align*}
|\{ {\text {copies of } \cF \text{  in  }} \Gamma^{(r)}(\mathcal{H}) \}| & \leq \sum_{h \in E(\mathcal{H})} 
\binom{|h|}{r+1} n^{|V(\cF)|-r-1}|V(\cF)| ! \\
& \leq \sum_{h \in E(\mathcal{H})} 
{ |h|^{r+1}} n^{|V(\cF)|-r-1}|V(\cF)| ! \\
& = |V(\cF)|!n^{|V(\cF)|-r-1}\sum_{h \in E(\mathcal{H})} |h||h|^{r} 
\\
& =  |V(\cF)|!n^{|V(\cF)|-r-1}\cdot o(n)\sum_{h \in E(\mathcal{H})} |h|^r 
\\
& = |V(\cF)|!n^{|V(\cF)|-r-1}\cdot o(n) \cdot O(n^r)
= o (n^{|V(\cF)|}). 
\end{align*}

Here we used the assumption $|h| = o(n)$ as well as (\ref{part1}).
\end{proof}

By the hypergraph removal lemma and Claim~\ref{claim2}, there exists a collection $\mathcal{R}$ of $o(n^r)$ $r$-edges such that each copy of $\cF$ in $\Gamma^{(r)}(\mathcal{H})$ contains at least one $r$-edge from $\mathcal{R}$.
Let $\cR' \subseteq \cR$ be the $r$-edges in $\cR$ that are blue, i.e., that are contained in at most $|E(\cF)|-1$ hyperedges of $\cH$.
We now show that for any hyperedge $h \in E(\mathcal{H})$, that every subset $S\subseteq h$ of size $R^{(r)}(\cF,\cF)$ contains an $r$-edge from $\mathcal{R}'$.
Suppose otherwise, then every $r$-edge in $\Gamma^{(r)}(S) \cap \mathcal{R}$ is not blue, i.e.,  is in at least $|E(\cF)|$ hyperedges of $\cH$. 

Let us color the $r$-edges in $\Gamma^{(r)}(S) \backslash \mathcal{R} $ red and the $r$-edges in $\Gamma^{(r)}(S) \cap \mathcal{R}$ green to get a $2$-coloring of the $r$-edges of $\Gamma^{(r)}(S)$. Then there is a monochromatic copy of $\cF$ in $\Gamma^{(r)}(S)$, as $|S|=R^{(r)}(\cF,\cF)$. 
Every copy of $\cF$ in $\Gamma^{(r)}(\mathcal{H})$ has a hyperedge in $\cR$, so this monochromatic copy cannot be red; hence it is green. Thus every hyperedge of this copy of $\cF$ is contained in at least $|E(\cF)|$ hyperedges of $\cH$. This implies that we can find a Berge-$\cF$ in $\mathcal{H}$ by choosing, for each $r$-edge $f$ of $\cF$ a unique hyperedge of $\cH$ that contains $f$; a contradiction.
Therefore, for any subset $S$ of size $R^{(r)}(\cF,\cF)$ in a hyperedge $h \in E(\mathcal{H})$ contains an $r$-edge from $\mathcal{R}'$.

\begin{clm}\label{claim3}
There exists a constant $ \beta = \beta(r,|V(\cF)|) < 1 $ such that for any hyperedge $h \in E(\mathcal{H})$, the number of $r$-edges of $\cR'$ in $\Gamma^{(r)}(h)$ is at least $(1-\beta)\binom{|h|}{r}$.
\end{clm}
\begin{proof}[Proof of Claim~\ref{claim3}]
Fix a hyperedge $h \in E(\cH)$ and a subset $S \subseteq h$ of size $R^{(r)}(\cF,\cF)$. By the argument above, $\Gamma^{(r)}(S)$ contains an $r$-edge from $\cR'$. Therefore, $\Gamma^{(r)}(h) \backslash \mathcal{R}'$ does not contain a clique of size $R^{(r)}(\cF,\cF)$.
So, by Theorem~\ref{dec}, there is a constant $\beta = \beta(r,|V(\cF)|) < 1$ such that $\Gamma^{(r)}(h) \backslash \mathcal{R}'$ contains at most $\beta \binom{|h|}{r}$ $r$-edges. Thus $\Gamma^{(r)}(h) \cap \mathcal{R}'$ contains at least $(1-\beta) \binom{|h|}{r}$ $r$-edges.
\end{proof}

By Claim~\ref{claim3},
\begin{align*}
\sum_{h \in E(\mathcal{H})} (1-\beta) \binom{|h|}{r} & \leq \sum_{h \in E(\mathcal{H})}\left( \Gamma^{(r)}(h) \cap \mathcal{R}' \right) \\ & \leq |\mathcal{R}'| \ (|E(\cF)|-1) \\
& \leq |\mathcal{R}| \ (|E(\cF)|-1) = o(n^r). \\
\end{align*}
 This implies that
\[
\sum_{h\in E(\mathcal{H})} |h|^r = o(n^r),
\]
which completes the proof of Theorem~\ref{thm1}.
\end{proof}

We finish this section with the proof of Corollary~\ref{cor1} (which immediately implies Corollary \ref{cor2}).

\begin{proof}[Proof of Corollary \ref{cor1}] Let us partition $\cH$ into $\cH_1\cup\cH_2$, where $\cH_1$ consists of the hyperedges of size at most $\sqrt{n}$ and $\cH_2$ consists of the hyperedges of size greater than $\sqrt{n}$. Then $\sum_{h\in E(\cH_1)}w(|h|)=o(n^r)$ by Theorem \ref{thm1}, using the assumption that every hyperedge of $\cH_1$ has size at most $\sqrt{n}=o(n)$ and at least $R^{(r)}(\cF,\cF)$.

On the other hand, as  $w(m) \leq o(m^r)$ we have
\[
\sum_{h\in E(\cH_2)}w(|h|)\le \sum_{h\in E(\cH_2)}o(|h|^r)=o\left(\sum_{h\in E(\cH_2)}|h|^r\right)=o(n^r).
\]
Adding these two bounds completes the proof.
\end{proof}

\section{Graph based hypergraphs}

In this section we consider the problem of estimating the maximum number of hyperedges in a Berge-$\cF$-free $k$-graph if the $r$-graph $\cF$ itself is based on a (hyper)graph. First we observe that it is possible to extend a hypergraph $\cF$ to a Berge$_k$-$\cF$ in multiple steps.

\begin{observation}
 For an $r$-graph $\cF$ and $k\ge l\ge r$ we have that the family of hypergraphs in Berge$_k$-$\cF$ is isomorphic to the family of hypergraphs in Berge$_k$-Berge$_l$-$\cF$.
\end{observation}

This simple observation implies bounds when $\cF$ is one particular Berge copy of a graph $F_0$. 
An obvious choice is the \textit{expansion} $F_0^{+r}$ of $F_0$. Recall that $F_0^{+r}$ is obtained by adding $r-2$ new vertices to each edge of $F_0$ such that these new vertices are distinct for distinct edges (thus all the intersections of the hyperedges are inherited from the graph). Tur\'an problems for expansions have been widely studied; see \cite{mubver} for a survey. 

Obviously, for $k \leq r$, the hypergraph $F_0^{+r}$ is a Berge copy of $F_0^{+k}$. On the other hand, every Berge$_r$-$F_0^{+k}$ is a Berge copy of $F_0$. Thus 

\begin{proposition}\label{expan}
Let $F_0$ be a fixed graph and let $k \leq r$.
Then
\[
\ex_r(n,\textup{Berge}_r\text{-}F_0)\le \ex_r(n,\textup{Berge}_r\text{-}F_0^{+k})\le \ex_r(n,F_0^{+r}).
\]
\end{proposition}

In particular, for $F_0=K_m$, with $m>r$ and for $n$ large enough, we know that the upper and lower bounds in the above proposition coincide, using a result of Pikhurko \cite{pikhu} and the following construction. Let us partition an $n$-element set $V$ into $m-1$ parts such that the difference between the size of two parts is at most 1 and let the hyperedges be the $r$-sets that intersect every part in at most one vertex. This $r$-uniform hypergraph is called the \textit{Tur\'an hypergraph} and is denoted by $T^r(n,m-1)$. It is Berge-$K_m$-free, as a $K_m$ would have two vertices $u$ and $v$ in the same part by the pigeonhole principle, but there is no hyperedge containing both $u$ and $v$.

Moreover, for any graph with chromatic number $m>r$, the upper bound is asymptotically the same as the lower bound (given by $T^r(n,m-1)$) in the above proposition due to a result of Palmer, Tait, Timmons, and Wagner \cite{pttw}. 

\begin{corollary}\label{expacor}
Let $F_0$ be a fixed graph and let $m=\chi(F_0)>r>k$. Then 
\[
 \ex_r(n,\textup{Berge}_r\text{-}F_0^{+k})=(1+o(1))|E(T^r(n,m-1)|.
\]
\end{corollary}

However, note that the bounds in Proposition \ref{expan} are far from each other for any given $F_0$, if $r$ is large enough. Indeed, $\ex_r(n,\textup{Berge}_r\text{-}F_0)=O(n^2)$ by a result of Gerbner and Palmer~\cite{gp1} while $\ex_r(n,F_0^{+r})=\Omega(n^{r-1})$ if $F_0$ is not a star (by having all the $r$-sets containing a fixed vertex), and  $\ex_r(n,F_0^{+r})=\Omega(n^{r-2})$ if $F_0$ is a star (by having all the $r$-sets containing two fixed vertices).

Note that similar statements can be obtained for other specific Berge copies of $F_0$. There are several different kinds of graph based hypergraphs (i.e., ways to extend a graph to a hypergraph) studied in \cite{gp}, and each of them results in specific Berge copies of $F_0$.

\section{Generalized hypergraph Tur\'an problems}\label{berge-section}

In this section we will prove Lemma \ref{celebhyp}. Our proof is based on the proof of Lemma \ref{celeb2}.
We use the following lemma from \cite{gerbner} (most of which already appears in \cite{gmp}).%, but the last part implies a small strengthening.

\begin{lemma}\label{cel2} 
Let $G$ be a finite bipartite graph with parts $A$ and $B$ and let $M$ be a maximum matching in $G$. Let $B'$ denote the set of vertices in $B$ that are incident to $M$. Then we can partition $A$ into $A_1$ and $A_2$ and partition $B'$ into $B_1$ and $B_2$ such that the vertices of $A_1$ are joined to the vertices of $B_1$ by edges of $M$ and every neighbor of the vertices of $A_2$ is in $B_2$. Moreover, every vertex in $A_1$ has a neighbor in $B\setminus B'$.
\end{lemma}

\begin{proof}[Proof of Lemma \ref{celebhyp}] Let $\cH_0$ be a Berge-$\cF$-free $r$-graph on $n$ vertices with the largest number of $r$-edges and consider the following auxiliary bipartite graph $G$. Part $A$ consists of the hyperedges of $\cH_0$ and part $B$ consists of their $k$-shadows.
%, i.e. the $k$-subsets of $V(H)$ that are contained in at least one hyperedge of $H$.
A vertex $u\in A$ is adjacent to a vertex $v\in B$ if and only if $v\subseteq u$. Let $M$ be a maximum matching in $G$ and let $B'$ denote the set of vertices in $B$ that are incident to $M$. By Lemma~\ref{cel2}, we can partition $A$ into $A_1$ and $A_2$ and partition $B'$ into $B_1$ and $B_2$ such that the vertices of $A_1$ are joined to the vertices of $B_1$ by edges of $M$ and every neighbor of the vertices of $A_2$ is in $B_2$.
Let $\cH$ be the red-blue $k$-graph with red hyperedges $B_1$ and blue hyperedges $B_2$.  Then
\[
\ex_r(n,\textup{Berge-}\cF)=|E(\cH_0)|=|A_1|+|A_2|\le |B_1|+|A_2|\le |B_1|+N(\cK_r^{(k)},\cH_{blue})=g_r(\cH).
\]
For the first inequality we used that that $M$ joins vertices of $A_1$ to $B_1$ and for the second inequality we used that that for a vertex $v\in A_2$, all its neighbors are in $B_2$, which means all its $k$-subsets are in $B_2$, thus there is a $\cK_r^{(k)}$ on those $r$ vertices.
\end{proof}

\begin{corollary}\label{hypceleb} For any $k$-graph $\cF$ and integers $r,n$ we have
\[
\ex_k(n,\cK_r^{(k)},\cF)\le \ex_r(n,\textup{Berge-}\cF)\le \ex_k(n,\cK_r^{(k)},\cF)+\ex_k(n,\cF).
\]
\end{corollary}

Note that the lower bound is given by replacing each $\cK_r^{(k)}$ in an $\cF$-free $k$-graph by an $r$-edge. Also note that Corollary~\ref{hypceleb} is very useful for graph based Berge hypergraphs, given that there are several results concerning the Tur\'an and the generalized Tur\'an problems. In the hypergraph case, much less is known about the Tur\'an numbers. The situation is even worse for the generalized hypergraph Tur\'an problems, where we are aware of only a few such results due to Ma, Yuan and Zhang~\cite{myz} %who proved bounds in case $\cF$ is a complete $r$-partite $r$-graph, with one of the parts being very large. 
and Xu, Zhang and Ge~\cite{XZG}.
Therefore, Corollary~\ref{hypceleb} is more useful in giving bounds for the 
  generalized hypergraph Tur\'an problem, i.e., $\ex_k(n,\cH,\cF)$. For example, combining Corollary~\ref{expacor} and Corollary~\ref{hypceleb} with the fact that $\ex_k(n,F_0^{+k}) = O(n^k) =o(n^r)$ gives

\begin{corollary}\label{genhyptur} Let $F_0$ be a graph and let $m=\chi(F_0)>r>k$. Then 
\[
\ex_k(n,\cK_r^{(k)},F_0^{+k})=(1+o(1))|E(T^r(n,m-1)|.
\]
\end{corollary}

%Indeed, we use Corollary~\ref{hypceleb}. Observe that $ \ex_k(n,\cK_r^{(k)},F_0^{+k})=\Omega(n^r)$ (shown by $T^r(n,m-1)$), while $\ex_k(n,F_0^{+k})=O(n^k)$ by definition. Thus the difference between the upper and lower bound in Corollary \ref{hypceleb} is negligible, thus they have the same asymptotic value as $\ex_r(n, \text{Berge-}F_0^{+k})$.

\section{Concluding remarks}

Let us note that one can define Berge copies of other discrete structures in a similar manner to hypergraphs. Assume that $\cF$ is a subset of an underlying set $X$ and there is a set $Y$ with a partial relation 
$\leq$ between elements of $X$ and $Y$. Then a \emph{Berge copy of $\cF$} is obtained by replacing every element $x$ of $\cF$ by an element of $y\in Y$ with $x\le y$ in such a way that we replace the elements of $X$ with distinct elements of $Y$. In this paper, $X$ is the system of $k$-subsets of a finite set, $Y$ consists of subsets of size at least $k$ of a potentially larger underlying finite set, and the relation $\leq$ corresponds to inclusion. To give another example, Anstee and Salazar \cite{anssal} considered 0-1 matrices with the relation being ``larger or equal in every entry''. 

There are many other settings where defining Berge copies makes sense, including permutations, vertex ordered graphs, edge ordered graphs, Boolean functions, et cetera.
Lemma~\ref{cel2} can be used in all these settings to connect the problem of counting the elements of $Y$ in a Berge-$\cF$-free subset of $Y$ to the following problem: what is the maximum number of elements $y$ of $Y$ such that every $x\in X$ with $x\le y$ belongs to $S$, where $S$ is an $\cF$-free subset of $X$?


\begin{thebibliography}{99}

\bibitem{as} N. Alon and C. Shikhelman. Many T copies in H-free graphs. \textit{Journal of Combinatorial
Theory, Series B}, \textbf{121}:146--172, 2016.

\bibitem{alsh} N. Alon and A. Shapira. On an extremal hypergraph problem of Brown, Erd\H ps and S\'os. \textit{Combinatorica}, \textbf{26}(6), 627--645, 2006.


\bibitem{anssal} R. Anstee and S. Salazar. Forbidden Berge hypergraphs. Electronic
Journal of Combinatorics, 24(1), 2017. P1.59.

\bibitem{agy} M. Axenovich, A. Gy\'arf\'as. A note on Ramsey numbers for Berge-$G$ hypergraphs,  arXiv preprint arXiv:1807.10062, 2018.

\bibitem{berge} C. Berge. Graphs and hypergraphs. North-Holland Pub. Co., 1973.

\bibitem{deca} D. de Caen. Extension of a theorem of Moon and Moser on complete subgraphs. Ars Combinatoria, 16, 5--10, 1983.

\bibitem{sgmp} S. English, D. Gerbner, A. Methuku, C. Palmer. On the weight of Berge-$F$-free hypergraphs, \textit{manuscript}

\bibitem{linsat} S. English, D. Gerbner, A. Methuku and M. Tait. Linearity of Saturation
for Berge Hypergraphs, arXiv preprint arXiv:1807.06947 (2018).

\bibitem{egkms} S. English, N. Graber, P. Kirkpatrick, A. Methuku, E.C. Sullivan,  Saturation of Berge hypergraphs. (2017). arXiv preprint arXiv:1710.03735.

%\bibitem{f} Z. F\"uredi. Linear trees in uniform hypergraphs. European Journal of Combinatorics 35 (2014): 264-272.

\bibitem{fkl} Z. F\"uredi, A. Kostochka, R. Luo. Avoiding long Berge cycles. \textit{Journal of Combinatorial Theory, Series B}, to appear

%\bibitem{fo} Z. F\"uredi,  L. \"Ozkahya. On 3-uniform hypergraphs without a cycle of a given length. Discrete Applied Mathematics 216 (2017): 582-588.

\bibitem{gerbner} D. Gerbner.  A note on the Tur\'an number of a Berge odd cycle. {\it arXiv preprint} arXiv:1903.01002 (2019).

%\bibitem{ggymv} D. Gerbner, E. Gy\H{o}ri, A. Methuku, M. Vizer. Generalized Tur\'an problems for even cycles. \textit{arXiv preprint} arXiv:1712.07079.

\bibitem{gmov} D. Gerbner, A. Methuku, G. Omidi, M. Vizer Ramsey problems for Berge hypergraphs. (2018) arXiv preprint arXiv:1808.10434.

\bibitem{gmp} D. Gerbner, A. Methuku, C. Palmer. General lemmas for Berge-Tur\'an hypergraph problems.  {\it arXiv preprint} arXiv:1808.10842 (2018).

\bibitem{gmv} D. Gerbner, A. Methuku, and M. Vizer. Asymptotics for the Tur\'an number of Berge-$K_{2,t}$. \textit{Journal of Combinatorial Theory, Series B}, to appear

\bibitem{gp1} D. Gerbner, C. Palmer. Extremal Results for Berge Hypergraphs. {\it SIAM Journal on Discrete Mathematics}, \textbf{31}.4: 2314--2327, 2017.

\bibitem{gp2} D. Gerbner, C. Palmer. Counting copies of a fixed subgraph in $ F $-free graphs. {\it arXiv preprint} arXiv:1805.07520 (2018).

\bibitem{gp} D. Gerbner, B. Patk\'os. Extremal Finite Set Theory,
1st Edition, CRC Press, 2018.

\bibitem{GMT}
D. Gr\'osz, A. Methuku, C. Tompkins. Uniformity thresholds for the asymptotic size of extremal Berge-$F$-free hypergraphs. arXiv preprint arXiv:1803.01953 (2018).

\bibitem{gyori} E. Gy\H ori. Triangle-free hypergraphs. \textit{Combinatorics, Probability and Computing}, 15(1-2), 185--191, 2006.

\bibitem{gykl} E. Gy\H ori, G. Y. Katona, N. Lemons. Hypergraph extensions of the Erd\H os-Gallai theorem. \textit{European Journal of Combinatorics}, 58, 238--246, 2016.

\bibitem{gyl} E. Gy\H ori, N. Lemons. Hypergraphs with no cycle of a given length. \textit{Combinatorics,
Probability and Computing}, 21(1-2):193--201, 2012.

%\bibitem{gymstv} E. Gy\H{o}ri, A. Methuku, N. Salia, C. Tompkins, M. Vizer. On the maximum size of connected hypergraphs without a path of given length. \textit{Discrete Mathematics}, \textbf{341}(9), 2602--2605, 2018.

\bibitem{kllw} L. Kang, L. Liu, L. Lu, Z. Wang, The extremal $ p $-spectral radius of Berge-hypergraphs. (2018). arXiv preprint arXiv:1812.06032.

%\bibitem{kee} P. Keevash. Hypergraph turan problems. Surveys in combinatorics 392 (2011): 83--140.

\bibitem{lazver} F. Lazebnik, J. Verstra{\"e}te. On hypergraphs of girth five. The electronic journal of combinatorics, (2003) 10(1), 25.

\bibitem{myz} J. Ma, X. Yuan, and M. Zhang. Some extremal results on complete degenerate hypergraphs.
Journal of Combinatorial Theory, Series A, 154:598--609, 2018.

%\bibitem{muba} D. Mubayi. A hypergraph extension of Tur\'an's theorem. \textit{Journal of Combinatorial Theory, Series B}, \textbf{96}, 122–-134, 2006.

\bibitem{mubver} D. Mubayi, J Verstra\"ete. A survey of Tur\'an problems for expansions. Recent Trends in Combinatorics, 117--143, 2016. 

\bibitem{pttw} C. Palmer, M. Tait, C. Timmons, A.Z. Wagner, Tur\'an numbers for Berge-hypergraphs and related extremal problems. (2017). arXiv preprint arXiv:1706.04249.

\bibitem{pikhu} O. Pikhurko. Exact computation of the hypergraph Tur\'an function for expanded complete
2-graphs. \textit{Journal of Combinatorial Theory, Series B}, 103.2: 220--225, 2013.

\bibitem{STZZ} N. Salia, C. Tompkins, Z. Wang, O. Zamora. Ramsey numbers of Berge-hypergraphs and related structures. arXiv preprint arXiv:1808.09863, 2018.

\bibitem{t} C. Timmons, On $r$-uniform linear hypergraphs with no Berge-$K_{2, t}$. arXiv preprint arXiv:1609.03401 (2016).

\bibitem{XZG} Z. Xu, T. Zhang, G. Ge. Some extremal results on hypergraph Tur\'an problems.
arXiv preprint arXiv:1905.01685, 2019.

\bibitem{zykov} A. A. Zykov. On some properties of linear complexes. \textit{Matematicheskii sbornik},
\textbf{66}(2):163--188, 1949.


\end{thebibliography}
\end{document}